\documentclass[letterpaper, 10 pt,]{article}
\usepackage[left=3.5cm,right=4cm]{geometry}
\usepackage[utf8]{inputenc}
\usepackage[english]{babel}
\usepackage[T1]{fontenc}
\usepackage{amsthm}
\usepackage{amsfonts}

\usepackage{amsmath} % assumes amsmath package installed
\usepackage{amssymb}  % assumes amsmath package installed
\usepackage{color}
\usepackage[utf8]{inputenc}
\usepackage{graphicx}
\usepackage{braket} % allow for \Set{...}
\usepackage{algorithm}

\newtheorem{proposition}{Proposition}
\newtheorem{definition}{Definition}
\newtheorem{remark}{Remark}
\newtheorem{lemma}{Lemma}
\newtheorem{assumption}{Assumption}

\newcommand{\A}{\mathcal{A}}
\newcommand{\M}{\mathcal{M}}

\newcommand{\nom}{\mathrm{nom}}
\newcommand{\U}{\mathcal{U}}

\newcommand{\dee}{\mathrm{d}}

\newcommand{\OCPT}{\mathrm{OCP_T}}

\DeclareMathOperator*{\arginf}{arg\,inf}
\DeclareMathOperator*{\argmin}{arg\,min}

\usepackage{hyperref}

\title{\LARGE \bf
Model Predictive Control Tailored to Epidemic Models
}

\author{%
	Philipp Sauerteig%
	\footnote{Corresponding author.}\,
	\thanks{Technische Universit{\"{a}}t Ilmenau, Ilmenau, Germany, Institute for Mathematics %
		(\texttt{[philipp.sauerteig,mitsuru.wilson,karl.worthmann]@tu-ilmenau.de}).}
	\and
	Willem Esterhuizen%
	\thanks{Technische Universit\"at Chemnitz, Germany, Automatic Control and System Dynamics Laboratory %
		(\texttt{[willem.esterhuizen,stefan.streif]@etit.tu-chemnitz.de}).}
	\and
	Mitsuru Wilson\footnotemark[2]
	\and
	Tobias K. S. Ritschel%
	\thanks{Max Planck Institute for Dynamics of Complex Technical Systems, Magdeburg, Germany
		(\texttt{tobk@dtu.dk}).}
	\and
	Karl Worthmann\footnotemark[2]
	\and
	Stefan Streif\footnotemark[3]
}

\date{}

\begin{document}

\maketitle

\begin{abstract}
	We propose a model predictive control (MPC) approach for minimising the social distancing and quarantine measures during a pandemic while maintaining a hard infection cap.\ %
	To this end, we study the admissible and the maximal robust positively invariant set (MRPI) of the standard SEIR compartmental model with control inputs.\ %
	Exploiting the fact that in the MRPI all restrictions can be lifted without violating the infection cap, we choose a suitable subset of the MRPI to define terminal constraints in our MPC routine and show that the number of infected people decays exponentially within this set.\ %
	Furthermore, under mild assumptions we prove existence of a uniform bound on the time required to reach this terminal region (without violating the infection cap) starting in the admissible set.\ %
	The findings are substantiated based on a numerical case study.\ %
\end{abstract}

\section{Introduction}
As the ongoing COVID-19 pandemic demonstrates, it is important to understand how infectious diseases spread and how countermeasures may affect this spread.\ %
Evaluating the impact of countermeasures based on mathematical models is an active field of research.\ %
In particular, there is a large body of work applying optimal control to epidemiology.\ %
Early work on this topic includes~\cite{hethcote1973optimal}, where vaccination policies are derived for an SIR (susceptible-infective-removed) model via dynamic programming; and~\cite{sanders1971quantitative} which solves optimal control problems for an SIS (susceptible-infective-susceptible) model.\ %
Many papers apply optimal control to SIR models, \cite{HanD11, godara2021control, kirschner1997optimal, agusto2017optimal}.\ %
See also the survey~\cite{sharomi2017optimal}.\ %

Model predictive control (MPC) has also been applied to epidemiology, such as in the papers \cite{selley2015dynamic,kohler2018dynamic,watkins2019robust} which consider stochastic networked epidemic models.\ %
Recently, many papers have appeared on the modelling and control of COVID-19 via MPC, see for example \cite{bonnans2020optimal,kohler2020robust}.\ %
The papers \cite{ARONNA2021100437,perkins2020optimal,grundel2020much,MorBCetal20} consider MPC of COVID-19 models where the control is limited to non-pharmaceutical interventions, whereas the papers \cite{GrunHeyd21b,parino2021model} also consider vaccination strategies.\ %

\emph{Set-based methods} have also been applied to epidemiology with the specific goal of control design that not only eliminates a disease, but also respects \emph{hard infection caps}.\ %
In~\cite{de2016viable} the authors describe the \emph{admissible set} (also known as the \emph{viability kernel} in viability theory~\cite{aubin2011viability}) of a two-dimensional model of a vector-borne disease, such as dengue fever or malaria.\ %
The paper~\cite{esterhuizen2020maintaining} extended this study, showing that the malaria model's admissible set as well as the \emph{maximal robust positively invariant set} (MRPI) may be described via the \emph{theory of barriers}.\ %
In \cite{DonL13, EstAS20} it is shown that parts of the boundaries of these sets are made up of special integral curves of the system that satisfy a minimum-/maximum-like principle.\ %
The recent paper~\cite{ester_epidemic_management_2021} uses the theory of barriers to describe both sets for the well-known SIR and SEIR epidemic models with and without plant-model mismatch.\ %

The papers \cite{de2016viable,esterhuizen2020maintaining,ester_epidemic_management_2021} argue that to maintain an infection cap, the control should be chosen based on the location of the state with respect to these sets.\ %
If it is located in the MRPI, intervention measures may be relaxed and economic damage may be minimized.\ %
However, if the state is located in the admissible set then intervention measures need to be carefully considered to avoid a breach in the infection cap.\ %
Some other research concerned with applying set-based methods to epidemiology include work on the \emph{set of sustainable thresholds}, \cite{barrios2018sustainable}, and the paper~\cite{rashkov2021model}, which describes the viability kernel of an SIR model through the solution of an associated Hamilton-Jacobi equation.\ %

In this paper, we combine the results from the paper~\cite{ester_epidemic_management_2021} with model predictive control (MPC)~\cite{grune2017nonlinear}.\ %
We consider the SEIR model and impose a hard cap on the proportion of infective individuals %
as well as constraints on the contact rate and removal rate.\ % 
We first show that \emph{it is possible} to eliminate the disease from any initial state located in the admissible set while always satisfying the infection cap, and that this \emph{is always the case} if the initial value is contained in the system's MRPI.\ %
Based on our findings, we then construct terminal conditions ensuring that the terminal costs used in our MPC implementation are uniformly bounded and prove initial feasibility under mild assumptions.\ %

In Section~\ref{Sec_SEIR_model_and_sets}, we introduce the constrained epidemic model and present important facts regarding the system's admissible set and the MRPI.\ %
Section~\ref{sec_vanilla_MPC} presents the MPC implementation that takes advantage of these sets.\ %
In this section we also present the paper's main result (Proposition~\ref{prop_mpc_works}) ensuring initial feasibility and eradication of the disease if the prediction horizon is sufficiently long.\ %
A numerical case study illustrates our findings in Section~\ref{sec_numerics}.\ %

For an interval $\mathcal{I} \subseteq \mathbb{R}$ and a set $Y \subseteq \mathbb{R}^m$ we denote the set of all locally integrable functions from~$\mathcal{I}$ to~$Y$ by
\begin{align}
	\mathcal{L}_\mathrm{loc}^1(\mathcal{I}, Y) := \Set{\begin{matrix} f : \mathcal{I} \to Y \\ \text{measurable} \end{matrix} | 
		\begin{matrix}
			\forall \, K \subset \mathcal{I}, \text{ compact}: \\
			\int_K |f_i(x)| \, \mathrm{d} x < \infty \\
			\forall \, i \in \{1, \ldots, m\} 
		\end{matrix} 
	}. \notag 
\end{align}

\section{The SEIR model}\label{Sec_SEIR_model_and_sets}
In this section, we introduce the SEIR model and the system's admissible set as well as the MRPI.\ 

\subsection{System dynamics and constraints}
We consider the SEIR model~\cite{Het2000}
\begin{subequations}\label{SEIR_IVP}
	\begin{align}
		\dot{S}(t) 	& = - \beta(t) S(t) I(t), \quad S(0) = S_0 \label{SEIR_eq_1}\\
		\dot{E}(t) 	& = \beta(t) S(t) I(t) - \eta E(t), \quad E(0) = E_0 \label{SEIR_eq_2}\\
		\dot{I}(t) 	& = \eta E(t) - \gamma(t) I(t), \quad I(0) = I_0 \label{SEIR_eq_3}\\
		\dot{R}(t)	& = \gamma(t) I(t), \quad R(0) = R_0, 
	\end{align}
\end{subequations}
where $S(t)$, $E(t)$, $I(t)$, and $R(t)$ describe the fractions of people who are either \emph{susceptible}, \emph{exposed}, \emph{infectious}, or \emph{removed} at time $t \geq 0$ and $(S_0, E_0, I_0, R_0) \in [0,1]^4$ with $S_0 + E_0 + I_0 + R_0 = 1$ denotes the initial value.\ %
In this context, compartment~$R$ accounts for both recovery and death due to a fatality caused by the disease.\ %
The parameter~$\eta^{-1} > 0$ denotes the average incubation time in days.\ %
Furthermore, in the standard formulation of the SEIR model, $\beta \equiv \beta_\mathrm{nom} > 0$ is the rate of infectious contacts and $\gamma \equiv \gamma_\mathrm{nom} > 0$ is the removal rate.\ %
However, we take countermeasures into account by allowing~$\beta$ and~$\gamma$ to be time-varying control variables.\ %
In particular, a value $\beta(t) < \beta_\mathrm{nom}$ reflects a reduction of the rate of infectious contacts, e.g.\ via contact restrictions or hygiene measures, while $\gamma(t) > \gamma_\mathrm{nom}$ can be interpreted as quarantining, and, thus, removing infectious people.\ %
These considerations motivate our control constraints
\begin{align}
	\beta(t) & \in [\beta_{\min}, \beta_{\nom}], \; \gamma(t) \in [\gamma_{\nom}, \gamma_{\max}] \quad \forall \, t \geq 0 \label{SEIR_eq_4}
\end{align}
with $\beta_\mathrm{min} > 0$ and $\gamma_\mathrm{max} < \infty$.\ %
Moreover, we aim to maintain a hard infection cap by satisfying the state constraint
\begin{align}
	I(t) \; \leq \; I_{\max} \label{SEIR_eq_5}
\end{align}
for some $I_\mathrm{max} \in (0,1)$.\ %

Note that~$R$ does not affect the other compartments.\ %
Thus, we may ignore it and analyse the three-dimensional system \eqref{SEIR_eq_1}--\eqref{SEIR_eq_3} under input and state constraints \eqref{SEIR_eq_4} and \eqref{SEIR_eq_5}.\ %
To ease our notation, we use $x(t) = (x_1(t), x_2(t), x_3(t))^\top := (S(t), E(t), I(t))^\top \in \mathbb{R}^3$ to denote the state and $u(t) := (\beta(t), \gamma(t))^\top \in \mathbb{R}^2$ to denote the control input at time $t \geq 0$ and write $\dot x(t) = f(x(t),u(t))$ short for~\eqref{SEIR_eq_1}--\eqref{SEIR_eq_3}.\ % 
Furthermore, we denote the set of feasible control values by $U := [\beta_{\min},\beta_{\nom}] \times [\gamma_{\nom}, \gamma_{\max}] \subset \mathbb{R}^2$ and the set of feasible controls as $\mathcal{U} = \mathcal{L}_\mathrm{loc}^1([0, \infty), U)$.\ %
Thus, for any initial value~$x_0$ and control input $u \in \mathcal{U}$ there exits a unique solution of the initial value problem~\eqref{SEIR_IVP}.\ %
We then write $x(\cdot; x_0, u)$ to denote the trajectory with respect to the three compartments~$S$, $E$, and~$I$ and highlight the dependence on~$x_0$ and~$u$.\ %
Based on $g : \mathbb{R}^3 \to \mathbb{R}$, $x \mapsto x_3 - I_\mathrm{max}$, the set of feasible states is given by $G := \Set{x \in \mathbb{R}^3 | g(x) \leq 0}$ and its interior by~$G_-$.\ %
Furthermore, the state trajectory satisfies $S(t) + E(t) + I(t) \leq 1$ for all $t \geq 0$.\ %
In other words, the system is confined to the positively invariant set
\[
\Pi := \Set{(S, E, I) \in [0,1]^3 | S + E + I \leq 1}.
\]

Given an initial state, $x_0 \in G_{\Pi} := G \cap \Pi$, we aim to determine a control~$u = u(x_0)$ that eliminates the disease, i.e., $\lim_{t \to \infty} (E(t) + I(t)) = 0$, while maintaining the hard infection cap, i.e., $I(t) \leq I_{\max}$ for all $t \in [0, \infty)$.\ %
To this end, we make use of the admissible set and the MRPI.\ %

\begin{remark}
	For $\gamma > \gamma_\mathrm{nom}$, i.e., if quarantine measures are active, a certain (additional) proportion of the infectious people are moved to compartment~$R$, see~\eqref{SEIR_eq_3}.\ % 
	Consequently, constraint~\eqref{SEIR_eq_5} does not reflect a bound on the actual number of infectious people, but rather on those, who are able to infect others.\ % 
	However, it suffices to illustrate our approach.\ % 
	For a more elaborate way to enforce an infection cap, we refer to~\cite{grundel2020much,GrunHeyd21b}.\ % 
\end{remark}

\subsection{Admissible and maximal robust positively invariant set}
The admissible set contains all states, for which there exists an input such that the state constraints are satisfied for all future time.\ %
\begin{definition}
	The \emph{admissible set} for the system~\eqref{SEIR_IVP}--\eqref{SEIR_eq_5} is given by
	\begin{align*}
		\mathcal{A} := \Set{ x_0 \in G_{\Pi} | \exists u\in\mathcal{U}:  x(t;x_0,u) \in G_{\Pi} \, \forall \, t \in [0,\infty)}.
	\end{align*}
\end{definition}
Moreover, we consider the set of states for which any control ensures feasibility.\ %
\begin{definition}
	The \emph{maximal robust positively invariant set} (MRPI) contained in $G_{\Pi}$ of the system \eqref{SEIR_IVP}--\eqref{SEIR_eq_5} is given by
	\begin{align*}
		\mathcal{M} := \Set{ x_0 \in G_{\Pi} | x(t; x_0, u) \in G_{\Pi} \, \forall \, t \in [t_0, \infty) \, \forall u \in \mathcal{U}}.
	\end{align*}
\end{definition}

In~\cite{ester_epidemic_management_2021}, both sets have been characterized using the \emph{theory of barriers}~\cite{DonL13, EstAS20}.\ %
In particular, the sets for system~\eqref{SEIR_IVP} with constraints~\eqref{SEIR_eq_4}--\eqref{SEIR_eq_5} are compact and never empty with $\M \subseteq \A$.\ %

In order to eliminate the disease, we need to drive the state to the set of equilibria given by
\[
\mathcal{E} := \Set{(S, E, I) \in [0,1]^3 | S \in [0,1], E = 0, I = 0}.
\]
Clearly, $\mathcal{E} \subset \mathcal{M}$.\ %

The next lemma summarizes some facts about the SEIR model in terms of~$\mathcal{A}$ and~$\mathcal{M}$ that will be useful in the sequel.\ %
\begin{lemma}\label{lemma_on_sets}
	The following assertions hold true.
	\begin{enumerate}
		\item For every $x_0 \in \Pi$ and any $u \in \mathcal{U}$ the epidemic will die out asymptotically, i.e., the compartments satisfy $\lim_{t \to \infty} E(t) = \lim_{t \to \infty} I(t) = 0$ as well as $\lim_{t \to \infty} (S(t) + R(t)) = 1$. 
		\item For every $x_0\in\A$ there exists a $u\in\U$ such that $\lim_{t \to \infty} E(t) = \lim_{t \to \infty} I(t) = 0$ and $I(t)\leq I_{\max}$ for all $t\geq 0$.
		\item For every $x_0\in\M$, $\lim_{t \to \infty} E(t) = \lim_{t \to \infty} I(t) = 0$ and $I(t)\leq I_{\max}$ for all $t\geq 0$, for any $u\in \U$.
	\end{enumerate}
\end{lemma}
\begin{proof}
	1) The idea is based on the proof of Theorem~5.1 in~\cite{Heth89}, where the assertion has been shown for the SIR model.\ %
	Let $x_0 \in \Pi$ and $u \in \mathcal{U}$ be arbitrary.\ %
	Since~$S$ and~$R$ are monotonic and bounded, the limits $S_\infty := \lim_{t \to \infty} S(t)$ and $R_\infty := \lim_{t \to \infty} R(t)$ exist with $\lim_{t \to \infty} \dot S(t) = \lim_{t \to \infty} \dot R(t) = 0$.\ %	
	Hence, 
	\begin{align}
		0 \; = \; \lim_{t \to \infty} \tfrac{1}{\gamma_\mathrm{max}} \dot R(t) \; \leq \; \lim_{t \to \infty} I(t) \; \leq \; \lim_{t \to \infty} \tfrac{1}{\gamma_\mathrm{min}} \dot R(t) \; = \; 0 \notag
	\end{align}
	and, therefore, $I_\infty := \lim_{t \to \infty} I(t) = 0$.\ %
	Furthermore, $E_\infty := \lim_{t \to \infty} E(t) = 1 - (S_\infty + I_\infty + R_\infty) \in [0,1]$ exists.\ %
	Assume $E_\infty > 0$.\ %
	Then, there exists some $t_0 \geq 0$ such that $\dot I(t) = \eta E(t) - \gamma(t) I(t) > 0$ for all $t \geq t_0$, in contradiction to $I_\infty = 0$.\\ %
	2) Since $x_0 \in \mathcal{A}$, there exists some feasible control $u \in \mathcal{U}$ such that $x(t; x_0; u) \in G_\Pi$ for all $t \geq 0$.\ %
	Also, noting that $\A\subseteq \Pi$, the assertion follows immediately from~1).\\ %
	3)  Since $x_0 \in \mathcal{M}$, $x(t; x_0; u) \in G_\Pi$ for all $t \geq 0$ for any $u \in \mathcal{U}$.\ %
	Also, noting that $\M \subseteq \Pi$, the assertion follows immediately from~1).
\end{proof}

In Section~\ref{sec_vanilla_MPC}, the set
\begin{align}
	\mathbb{X}_f := \Set{x \in \Pi | S \leq \tfrac{\gamma_\mathrm{nom}}{\beta_\mathrm{nom}}, \, I \leq I_\mathrm{max}, E \leq \tfrac{\gamma_\mathrm{nom}}{\eta} I_\mathrm{max}} \cup \mathcal{E} \notag 
\end{align}
will be used to define terminal constraints within MPC.\ %
We state the following lemma to motivate our choice.\ %
\begin{lemma}\label{lem:exponential_decay_if_S_suff_small}
	The set~$\mathbb{X}_f$ is contained in~$\mathcal{M}$.\ %
	Moreover, the compartments~$E$ and~$I$ decay exponentially for all $x_0 \in \mathbb{X}_f$ and all $u \in \mathcal{U}$.\ %
\end{lemma}
\begin{proof}
	Clearly, any point~$x_0$ with $S_0 \in [0,1]$ and $E_0 = I_0 = 0$ is in~$\mathcal{M}$.\ %
	Let $x_0 \in \mathbb{X}_f$ with $E_0 + I_0 > 0$ be given and consider $u = u_\mathrm{nom}$.\ %
	First, note that if $I_0 = I_\mathrm{max}$, then 
	\begin{align}
		\dot I(0) \; = \; \eta E(0) - \gamma_\mathrm{nom} I_\mathrm{max} \; \leq \; 0 \notag
	\end{align}
	by definition of~$\mathbb{X}_f$.\ %
	Moreover, if $I_0 = 0$, then $\dot I(0) > 0$ and $\dot E(0) < 0$.\ %
	Due to continuity, there exists some (sufficiently small) time $\delta > 0$ such that $I(\delta) > 0$ and $x(\delta) \in \mathbb{X}_f$.\ %
	Since~$S$ decreases monotonically, the compartments~$E$ and~$I$ are bounded from above by the solution of the initial value problem
	\begin{subequations}\label{subeq:linear_ode_from_above}
		\begin{align}
			\dot e(t) \; & = \; \beta S(\delta) i(t) - \eta e(t), \quad e(0) = E(\delta), \\
			\dot i(t) \; & = \; \eta e(t) - \gamma i(t), \quad i(0) = I(\delta). 
		\end{align}
	\end{subequations}
	In particular, 
	\begin{align*}
		\dot E(\delta + t) \; & \leq \; \dot e(t) \; \leq \; \beta_\mathrm{nom} S(\delta) i(t) - \eta e(t) \\
		\dot I(\delta +t ) \; & \leq \; \dot i(t) \; \leq \; \eta e(t) - \gamma_\mathrm{nom} i(t)
	\end{align*}
	holds.\ % 
	The solution of~\eqref{subeq:linear_ode_from_above} with $u = u_\mathrm{nom}$ is given by
	\begin{align}
		(e(t), i(t))^\top \; = \; \mathrm{e}^{A t} (E(\delta), I(\delta))^\top \notag
	\end{align}
	with matrix
	\begin{align}
		A \; = \; 
		\begin{bmatrix}
			- \eta & \beta S(\delta) \\
			\eta & - \gamma
		\end{bmatrix}.\notag
	\end{align}
	The eigenvalues of~$A$ are 
	\begin{align}
		\lambda_{1,2} \; = \; - \frac {\eta + \gamma} {2} \pm \sqrt{ \frac{(\eta + \gamma)^2}{4} - \eta \gamma + \eta \beta S(\delta) }. \notag
	\end{align}
	Hence, $S(\delta) < S_0 \leq \gamma / \beta$ yields exponential stability of~$A$.\ %
	Consequently, $e$ and~$i$ and, thus, $E$ and~$I$ decay exponentially.\ %
\end{proof}

\begin{remark}\label{rem:threshold_S}
	As in the SIR model~\cite{Heth89}, $S(t) = \gamma(t) / \beta(t)$ is a threshold in the following sense
	\begin{align}
		\frac{\mathrm{d}}{\mathrm{d}t} (E(t) + I(t)) %
		& \begin{cases}
			> 0 \quad \text{if } S(t) > \gamma(t) / \beta(t), \\
			= 0 \quad \text{if } S(t) = \gamma(t) / \beta(t), \\
			< 0 \quad  \text{if } S(t) < \gamma(t) / \beta(t),
		\end{cases} \notag
	\end{align}
	i.e., as long as~$S$ is sufficiently big, the total number of \emph{infected} (exposed or infectious) people increases monotonically and decays otherwise.\ %
	In the following, we use the notation $\bar S = \gamma_\mathrm{nom} / \beta_\mathrm{nom}$ to denote the threshold for herd immunity, which is also the inverse of the basic reproduction number~\cite{Het2000}.\ %
\end{remark}

\section{MPC with terminal cost and set}\label{sec_vanilla_MPC}
We propose an MPC scheme for determining a feedback control that aims to minimise the required contact restrictions and quarantine measures while maintaining a hard infection cap, i.e., given an initial value $x_0 \in G_\Pi$ we are interested in solving the optimal control problem
\begin{subequations}\label{OCP_infinite_horizon}
	\begin{align}
		\inf_{u \in \mathcal{U}} \quad & J_{\infty}(x_0, u) := \int_0^\infty \ell(x(t; x_0, u), u(t))\,\, \mathrm{d}t \\
		\mathrm{s.t.} \quad & \dot x(t) = f(x(t), u(t)), \quad x(0) = x_0 \\
		& g(x(t)) \leq 0 \quad \forall \, t \geq 0
	\end{align}
	with stage costs $\ell : \mathbb{R}^3 \times \mathbb{R}^2 \to \mathbb{R}$, 
	\begin{equation}
		\ell(x,u) := E^2 + I^2 + (\beta - \beta_{\nom})^2 + (\gamma - \gamma_{\nom})^2. \label{quadtratic_stage_cost}
	\end{equation}
\end{subequations}
For an introduction to MPC we refer to~\cite{grune2017nonlinear}.\ %

\subsection{MPC formulation}
We present a solution to the stated problem using continuous-time MPC.\ %
Thus, we consider the following finite-horizon optimal control problem,
\[
\OCPT : \quad \inf_{u \in \U_T^{\mathbb{X}_f}(x_0)} J_T(x_0,u),
\]
with $J_T(x_0, u) : \mathbb{R}^3 \times \mathcal{L}^1_\mathrm{loc}([0,T), \mathbb{R}^2) \to \mathbb{R}_{\geq 0}$, the cost functional, specified as
\[
J_T(x_0,u) := \int_0^T \ell(x(t; x_0, u), u(t))\,\dee t + J_{f}(x(T;x_0,u)),
\]
with terminal cost 
\begin{align}
	J_{f}(\hat{x}) & := \int_{0}^{\infty} x^2_2(t; \hat{x}, u_{\nom})  + x^2_3(t; \hat{x}, u_{\nom})\,\,\dee t \notag \\
	& = \int_{0}^{\infty} E_{\nom}^2(t; \hat x) + I_{\nom}^2(t; \hat x)\,\,\dee t. \notag
\end{align} 
Here, $E_{\nom}(\, \cdot \, ; \hat x)$ and~$I_{\nom}(\, \cdot \, ; \hat x)$ denote the exposed and infective compartments of the state trajectory obtained with the nominal input $u_\mathrm{nom} \equiv (\beta_{\nom}, \gamma_{\nom})^\top$ starting at~$\hat x$.\ %

Given an initial state $x_0 \in G_{\Pi}$ and a horizon length $T \in \mathbb{R}_{>0} \cup \{\infty\}$, the set $\U_T^{\mathbb{X}_f}(x_0)$ denotes all control functions $u \in \U$ for which $x(t;x_0,u) \in G_{\Pi}$ for all $t \in [0,T]$, and $x(T;x_0,u)\in \mathbb{X}_f$.\ %
The optimal value function is defined as $V_T(x_0) : \mathbb{R}^3 \to \mathbb{R}_{\geq 0} \cup \{+\infty\}$, 
\[
V_T(x_0):= \inf_{u\in\U_T^{\mathbb{X}_f}(x_0)} J_T(x_0,u).
\]
If $\U_T^{\mathbb{X}_f}(x_0) = \emptyset$, i.e., there does not exist a solution to $\OCPT$, we take the convention that $V_T(x_0) = \infty$.\ %
The continuous-time MPC algorithm is outlined in Algorithm~\ref{alg:mpc}.\ %

We assume that if $\mathcal{U}_T^{\mathbb{X}_f}(x_0)\neq \emptyset$ in step 1), then a minimizer is an element of $\U_T^{\mathbb{X}_f}(x_0)$, see \cite[p. 56]{grune2017nonlinear} for a discussion on this issue.\ %
Thus, given an initial state $x_0 \in \A$, $\OCPT$ is solved with the finite horizon $T = N \delta$, the first portion of the minimizer is applied to the system, the state is measured, and the process is iterated indefinitely.\ %
The algorithm implicitly produces the time-varying sampled-data MPC \emph{feedback}, $\mu_{T,\delta} : [0,\delta) \times G_{\Pi} \to U$, which results in the MPC \emph{closed-loop} solution, denoted $x_{\mu_{T,\delta}}(\cdot; x_0)$.\ %

\begin{algorithm}[h]
	\caption{Continuous-time MPC with target set}
	\noindent \textbf{Input}: $\delta \in \mathbb{R}_{>0}$, $N \in \mathbb{N}$, $x_0 \in \mathcal{A}$, $\mathbb{X}_f\subset\M$\\
	\textbf{Set} prediction horizon $T \leftarrow N \delta$. %
	\begin{enumerate}
		\item Find a minimizer $u^\star \in \arginf_{u \in \mathcal{U}_T^{\mathbb{X}_f}(x_0)} J_T(x_0, u)$.
		\item Implement $u^\star(\cdot)$, for $t \in [0, \delta)$. %
		\item Set $x_0 \leftarrow x(\delta; x_0, u^{\star})$ and go to step~1).
	\end{enumerate}
	\label{alg:mpc}
\end{algorithm}

\begin{remark}
	Given $u^{\star} = \argmin_{u\in\U_T^{\mathbb{X}_f}(x_0)} J_T(x_0,u)$, we denote $\hat{x}: = x(\delta; x_0, u^{\star})$ and define $\hat{u} : [0, T) \to U$ via $\hat{u}(t) = u^{\star}(\delta + t)$ for $t < T - \delta$ and $u(t) = u_\mathrm{nom}$ otherwise. %
	Then, $g(x(t; \hat x, \hat u)) \leq 0$ for all $t \geq 0$ and $x(T - \delta; \hat x, \hat u) \in \mathbb{X}_f$, i.e., $\hat{u} \in \mathcal{U}_T^{\mathbb{X}_f}(\hat{x})$.\ 
	Thus, initial feasibility of $\OCPT$ yields recursive feasibility. % 
\end{remark}

The following result is essential to prove that the MPC implementation works.\ %
\begin{lemma}\label{lemma_V_inf_bounded}
	Without countermeasures the infinite horizon cost functional is uniformly bounded on the MRPI, i.e.,\ %
	\begin{align}
		\exists \, C \in \mathbb{R}_{> 0} \; \forall \, x_0 \in \mathcal{M} : \; J_{\infty}(x_0, u_\mathrm{nom}) < C. \notag
	\end{align}
	Consequently, $V_{\infty}(x_0) < C$ for all $x_0 \in \mathcal{M}$.\ %
\end{lemma}

\begin{proof}
	Let $x_0 \in \mathcal{M}$ be given.\ %
	Then, any control $u \in \mathcal{U}$ yields $x(t; x_0, u) \in G_\Pi$ for all $t \geq 0$. Moreover, $I_\infty := \lim_{t \to \infty} I(t) = \lim_{t \to \infty} E(t) = 0$ and $S_\infty := \lim_{t \to \infty} S(t)$ and $R_\infty := \lim_{t \to \infty} R(t)$ exist.\ %
	Hence, 
	\begin{align}
		\int_0^\infty{\beta(t) S(t) I(t) \, \mathrm{d} t} = - \int_0^\infty{\dot S(t) \, \mathrm{d} t} = S_0 - S_\infty \in [0,1] \notag
	\end{align}
	and, therefore,
	\begin{align}
		\int_0^\infty{E(t) \, \mathrm{d} t} & = \tfrac 1 \eta \left( \int_0^\infty{\beta(t) S(t) I(t) \, \mathrm{d} t} - \int_0^\infty{\dot E(t) \, \mathrm{d} t} \right) \notag \\
		& = \tfrac 1 \eta ( S_0 - S_\infty + E_0). \notag
	\end{align}
	Moreover, since $E(t) \in [0,1]$, we get $E(t)^2 \leq E(t)$ for all $t \geq 0$ and, thus,
	\begin{align}
		\int_0^\infty{E(t)^2 \, \mathrm{d} t} \leq \int_0^\infty{E(t) \, \mathrm{d} t} = \tfrac 1 \eta ( S_0 - S_\infty + E_0) < \infty. \notag
	\end{align}
	Analogously, one can show that
	\begin{align}
		\int_0^\infty{I(t)^2 \, \mathrm{d} t} \leq \int_0^\infty{I(t) \, \mathrm{d} t} = \tfrac{1}{\gamma_\mathrm{nom}} (1 - R_0 - S_\infty) < \infty. \notag
	\end{align}
	In conclusion, we get
	\begin{align}
		V_\infty(x_0) \; \leq & \; J_\infty(x_0, u_\mathrm{nom}) \notag \\
		= & \; \int_0^\infty{E(t; x_0, u_\mathrm{nom})^2 + I(t; x_0, u_\mathrm{nom})^2 \, \mathrm{d} t} \notag \\
		\leq & \; \tfrac 1 \eta (S_0 - S_\infty + E_0) + \tfrac{1}{\gamma_\mathrm{nom}} (1 - R_0 - S_\infty) \notag \\
		\leq & \; \tfrac 2 \eta + \tfrac{1}{\gamma_\mathrm{nom}} < \infty, \notag
	\end{align}
	which completes the proof.\ %
\end{proof}

\subsection{Initial feasibility}
Next we show initial (and, thus, recursive) feasibility if the following mild assumption holds.\ %
\begin{assumption}\label{ass:all}
	If there are infected people, i.e., $E_0 + I_0 > 0$, then we assume the initial values to be bounded from below, i.e., there exists some $\varepsilon_0 > 0$ such that
	\begin{align}
		\varepsilon_0 \; \leq \; \min \{E_0, I_0\}. \label{ass:E0_I0>eps}
	\end{align}
	We further assume there exists some $K \in \mathbb{R}_{> 0}$ such that
	\begin{align}
		I(t) / (I(t) + E(t)) \; \geq \; K \quad \forall \, t \geq 0, \label{ass:ratio_I_to_E+I}
	\end{align}
	i.e., the fraction of infected people who are infectious is bounded from below.\ %
	In order to facilitate the proof of Lemma~\ref{lemma_reach_A_finite_time}, we further assume that 
	\begin{align}
		\beta_\mathrm{min} \; \leq \; \gamma_\mathrm{max}. \label{ass:beta<=gamma}
	\end{align}
	Note that this can always be achieved by allowing stricter social distancing or quarantine measures.\ %
\end{assumption}
Our numerical simulations in Section~\ref{sec_numerics} show that Assumption~\eqref{ass:E0_I0>eps} is essential to reach the terminal set in finite time, see Figure~\ref{fig:no_cost_controllability}.\ %
However, from a practical point of view, we are only interested in a situation where there already are infected people.\ %

Based on Assumption~\ref{ass:all}, we are able to show that the terminal set is reached in finite time.\ %
\begin{lemma}\label{lemma_reach_A_finite_time}
	Let assumption~\eqref{ass:ratio_I_to_E+I} hold.\ % 
	Then, the time required to reach the terminal set is uniformly bounded on the subset $\mathcal{A}^\prime := \Set{x_0 \in \mathcal{A} | \text{\eqref{ass:E0_I0>eps} holds}}$, i.e.
	\begin{align}
		\exists \, T \in \mathbb{R}_{> 0} \, \forall \, x_0 \in \mathcal{A}^\prime \, \exists \, u \in \mathcal{U} : x(T; x_0, u) \in \mathbb{X}_f. \notag
	\end{align}
	Hence, %
	\[
	\exists \, C \in \mathbb{R}_{> 0} \, \forall \, x_0 \in \mathcal{A}^\prime : \; V_{\infty}(x_0) < C.
	\]
\end{lemma}

\begin{proof}
	For $x_0 \in \mathcal{M}$, the assertion follows from Lemma~\ref{lemma_V_inf_bounded}.\ %
	Consider an arbitrary $x_0 \in \mathcal{A}^\prime \setminus \mathcal{M}$.\ %
	Then, $x_0 \notin \mathbb{X}_f$ according to Lemma~\ref{lem:exponential_decay_if_S_suff_small}.\ % 
	We construct a control $\tilde u \in \mathcal{U}$ such that $x(t; x_0, \tilde u) \in G_\Pi$ for all $t \geq 0$ and $x(T; x_0, \tilde u) \in \mathbb{X}_f$ for some finite $T > 0$.\ %
	To this end, we ensure a decay of~$S$ as follows.\ %
	
	Note that since $x_0 \in \mathcal{A}^\prime \setminus \mathcal{M}$ and $\lim_{t \to \infty} I(t) = 0$ there exists some control $u_1 \in \mathcal{U}$ satisfying $x(t; x_0, u_1) \in G_\Pi$ for all $t \geq 0$ and some time $t_1 \geq 0$ such that $I(t_1; x_0, u_1) = 1/2 \cdot I_\mathrm{max}$.\ %
	Moreover, based on assumption~\eqref{ass:beta<=gamma} we may assume that $\gamma_\mathrm{nom} / \beta_\mathrm{nom} < 1 < \gamma_\mathrm{max} / \beta_\mathrm{min}$.\ % 
	since otherwise, the pandemic would simply abate without control, see Remark~\ref{rem:threshold_S}.\ %
	Therefore, there exists some time $t_3 \in (t_1, \infty]$ and a control $u_2 = (\beta_2, \gamma_2)^\top \in \mathcal{U}$ such that
	\begin{align}
		\frac{\gamma_2(t)}{\beta_2(t)} = S(t) \in \left[ \frac{\gamma_\mathrm{nom}}{\beta_\mathrm{nom}},~\frac{\gamma_\mathrm{max}}{\beta_\mathrm{min}}\right] \quad \forall \, t \in [t_1, t_3] \notag
	\end{align}
	and, hence, the number of infected $E + I$ is constant on the interval $[t_1, t_3]$.\ %
	Since $\lim_{t \to \infty} (E(t)+I(t)) = 0$, this implies $t_3 < \infty$.\ %
	
	Next, we show that there exists some $t_2 \in [t_1, t_3]$ and a control $u_3 = (\beta_3, \gamma_3)^\top \in \mathcal{U}$ such that	 $E + I$ is constant on $t \in [t_2, t_3]$ and, in addition, 
	\begin{equation}
		-\dot E(t) = \dot I(t) = \eta E(t) - \gamma_3(t) I(t) = 0 \quad \forall \, t \in [t_2, t_3]. \notag
	\end{equation}
	Suppose $\eta E(t) / I(t) < \gamma_\mathrm{nom}$ for all $t > t_1$.\ %
	Then, $\dot E(t) > 0$ for all $t > t_1$, in contradiction to $\lim_{t \to \infty} E(t) = 0$.\ %
	On the other hand, if $\eta E(t) / I(t) > \gamma_\mathrm{max}$ for all $t > t_1$, then $\dot I(t) > 0$ for all $t > t_1$, in contradiction to $\lim_{t \to \infty} I(t) = 0$.\ %
	Therefore, $t_2$ and $u_3$ as above exist with 
	\begin{align}
		\gamma_3(t) = \eta E(t) / I(t) \quad \text{and} \quad \beta_3(t) = \eta E(t_2) / (S(t) I(t_2)) \notag
	\end{align}
	for all $t \in [t_2, t_3]$.\ %
	(The feasibility of~$\beta_3$ can be argued analogously to the one of~$\gamma_3$.)\ %
	
	Thus, we have
	\begin{align}
		\frac{\gamma_3(t)}{\beta_3(t)} = S(t) &= S(t_2) - I(t_2) \int_{t_2}^{t} \gamma_3(s) \mathrm{d}s \notag \\
		& \leq S(t_2) - I(t_2) \gamma_\mathrm{nom} (t - t_2) \notag
	\end{align}
	for all $t \in [t_2, t_3]$.\ %
	With the maximal~$t_3$ we arrive at $S(t_3) = \bar S$ 
	and, thus, $x(T, x_0, \tilde{u}) \in \mathbb{X}_f$ with $T = t_3$ and
	\begin{align}
		\tilde u(t) = 
		\begin{cases}
			u_1(t) \; & \text{if } t < t_1, \\
			u_2(t) \; & \text{if } t \in [t_1, t_2), \\
			u_3(t) \; & \text{if } t \in [t_2, t_3).
		\end{cases} \notag
	\end{align}
	Due to assumptions~\eqref{ass:E0_I0>eps} and~\eqref{ass:ratio_I_to_E+I}, the value 
	\begin{align}
		\sup_{x_0 \in \mathcal{A}^\prime} \inf \Set{t_1 \geq 0 | 
			\begin{matrix}
				\exists u \in \mathcal{U} : x(t; x_0, u) \in G_\Pi \, \forall \, t \geq 0 \\
				I(t_1; x_0, u) = 1/2 \cdot I_\mathrm{max}
		\end{matrix}} \notag
	\end{align}
	is bounded uniformly since 
	\begin{align}
		\dot S(t) \leq - \beta_\mathrm{min} K (E_0 + I_0) S(t) \leq - 2 \beta_\mathrm{min} K \varepsilon_0 S(t) \notag
	\end{align}
	for all $t \geq 0$ with $S(t) > \gamma_\mathrm{nom} / \beta_\mathrm{nom}$.\ %
\end{proof}

As argued in the proof of Lemma~\ref{lemma_reach_A_finite_time}, assumption~\eqref{ass:ratio_I_to_E+I} is not too restrictive.\ %
Whenever~$I(t)$ is small while~$E(t)$ is big at the same time, then $\dot I(t) > 0$, meaning that~$I$ is going to increase.\ %
Thus, after some time into the pandemic, there will always be a certain fraction of infected people who are infectious.\ %

Our main result states that the MPC feedback generated by Algorithm~\ref{alg:mpc} approximates the solution of~\eqref{OCP_infinite_horizon}.\ %
\begin{proposition}\label{prop_mpc_works}
	Let assumption~\eqref{ass:ratio_I_to_E+I} hold.\ % 
	There exists a finite prediction horizon $T > 0$ such that $\OCPT$~\eqref{OCP_infinite_horizon} is initially and, thus, recursively feasible for every $x_0 \in \mathcal{A}^\prime$ with $\lim_{t \to \infty}x_{\mu_T,\delta}(t,x_0) \in \mathcal{E}$, and $g(x_{\mu_T, \delta}(t,x_0)) \leq 0$ for all $t \geq 0$ under the MPC feedback~$\mu_{T,\delta}$ produced by Algorithm~\ref{alg:mpc}.\ %. 
\end{proposition}

\begin{proof}
	From the proof of Lemma~\ref{lemma_reach_A_finite_time}, the target set $\mathbb{X}_f$ is reachable from any $x_0 \in \A^\prime$ in finite time.\ %
	Thus, there exists a finite $T > 0$ such that for every $x_0 \in \A^\prime$ the problem $\OCPT$ is initially feasible.\ %
	Hence, $\OCPT$ is recursively feasible, implying that the infection cap is respected for all $t \geq 0$.\ %
	After the state reaches $\mathbb{X}_f \subset \M$, it approaches $\mathcal{E}$ asymptotically according to Lemma~\ref{lemma_on_sets}.\ %
\end{proof}
\begin{remark}
	For the SEIR model, the final costs~$J_f$ are in fact not needed to ensure stability, as is usual in MPC.\ %
	However, it does improve the closed-loop transient performance.
\end{remark}

\section{Numerical case study}\label{sec_numerics}
We now illustrate the paper's theoretical results with simulations in \texttt{Matlab}, using the script provided by \cite{grune2017nonlinear} to implement the MPC controller.\ %
We consider the constrained SEIR model~\eqref{SEIR_IVP} under constraints \eqref{SEIR_eq_4} and \eqref{SEIR_eq_5}.\ %
We take $\beta_{\nom} = 0.44$, $\gamma_{\nom} = 1/6.5$, $\eta = 1/4.6$ and $I_{\max} = 0.05$~\cite{GrunHeyd21b}.\  %
Furthermore, we take $\gamma_{\max} = 0.5$ and $\beta_{\min} = 0.22$.\ %
The point clouds of the boundaries of~$\M$ and~$\A$, indicated in Figure~\ref{fig_1}, are obtained as detailed in the paper \cite{ester_epidemic_management_2021}, that is, via backward numerical integration of the system with a special "extremal input'' from particular initial points on the state constraint $I = I_{\max}$.\ % 
\begin{figure}[h]
	\begin{center}\
		\includegraphics[width=0.8\columnwidth]{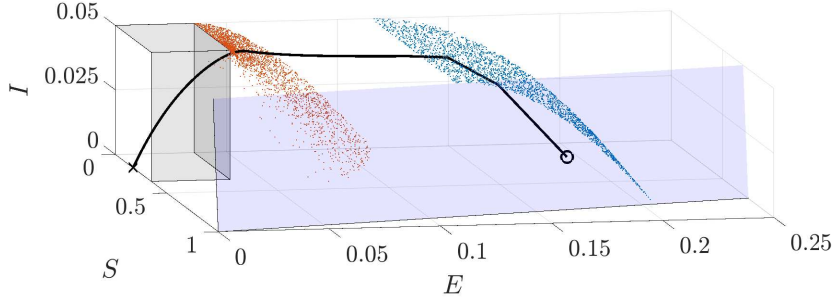} 
		\caption{The point clouds represent the boundaries of~$\mathcal{A}$ (blue) and~$\mathcal{M}$ (red) while the planes depict the boundaries of~$\mathbb{X}_f$ (grey) and~$\Pi$ (blue).\ %
			The trajectory obtained by the MPC feedback law starts in~$\mathcal{A}$ (black circle), approaches its boundary, reaches~$\mathcal{M}$, and converges to a disease-free equilibrium (black cross).}
		\label{fig_1}                             
	\end{center}                                
\end{figure}%
Also shown is a typical run of the MPC closed loop as produced by Algorithm~1 with $\delta = 1$.\ %
The prediction horizon is set to $T = 25$ days and the initial value $x_0 = (0.50, 0.18, 0.01)^\top$ is close to the boundary of~$\mathcal{A}$.\ %
The impact of imperfect predictions and state estimation is left for future research.\ % 

In Figure~\ref{fig:I_over_time} the evolution of~$I$ subject to the MPC feedback~$\mu_T$ is depicted.\ %
\begin{figure}[h]
	\centering
	\includegraphics[width=0.8\columnwidth]{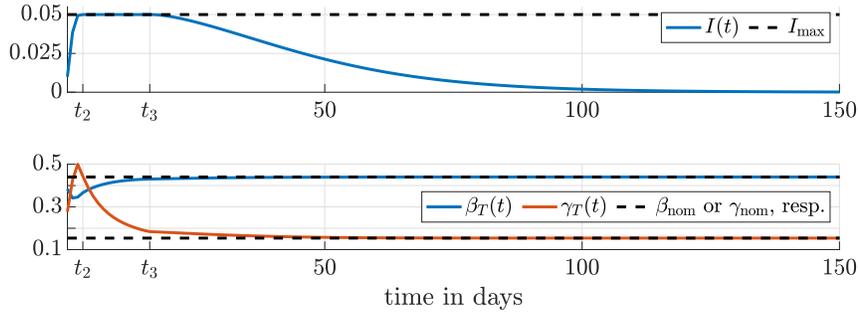} 
	\caption{Evolution of the infectious compartment~$I$ subject to the MPC feedback $\mu_T = (\beta_T, \gamma_T)^\top$ over time.\ %
		First, $I$ increases until $I(t_2) = I_\mathrm{max}$, then it stays there until $S(t_3) = \bar S$, and decays afterwards.}
	\label{fig:I_over_time}                                                       
\end{figure}
The trajectory follows the construction in the proof of Lemma~\ref{lemma_reach_A_finite_time}.\ %
First, $I$ increases until it reaches some controlled equilibrium %
$\dot I(t_2) = 0$ until herd immunity is achieved with $S(t_3) = \bar S$.\ %
Then, it decays asymptotically towards zero.\ %
Note that the optimal equilibrium in this context is the infection cap, i.e., $I \equiv I_\mathrm{max}$ since it maximises the descent of~$S$.\ %
Furthermore, we weight the cost terms in~\eqref{quadtratic_stage_cost} equally since tuning weights is out of scope of this paper.\ % 
As a result, the cost term with respect to~$\gamma_T$ is much bigger than the one associated with~$\beta_T$.\ %
Keep in mind that in our model, only infectious people are put into quarantine and we do not consider re-infections, i.e., they are completely removed from the system dynamics.\ %

Figure~\ref{fig:no_cost_controllability} motivates assumption~\eqref{ass:E0_I0>eps}.\ %
\begin{figure}[h]
	\centering
	\includegraphics[width=0.8\columnwidth]{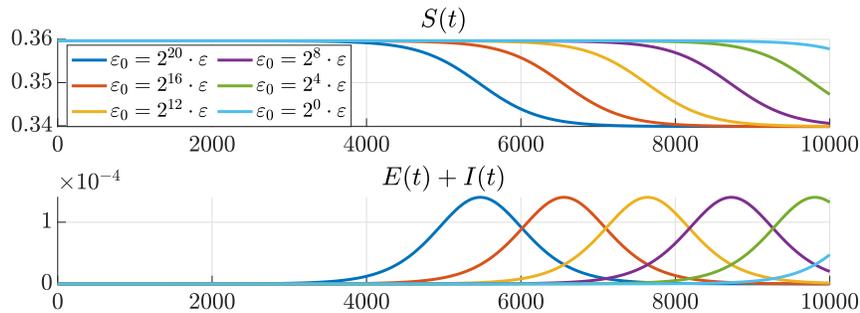}
	\caption{Impact of initial value for small $x_0 = (\bar{S} + 0.01, \varepsilon_0, \varepsilon_0)^\top$ without control, i.e., $u \equiv u_\mathrm{nom}$ and with machine precision $\varepsilon = 2.2204 \cdot 10^{-16}$.}
	\label{fig:no_cost_controllability}
\end{figure}
We observe that starting close to $S_0 = \bar S$ and decreasing~$E_0$ and~$I_0$, the time for a pandemic to break out explodes.\ %
Moreover, the overall cost $J_\infty(x_0, u_\mathrm{nom})$ seems to be independent of the choice of~$E_0$ and~$I_0$.\ %
One reason for that is the fact that $E+I$ increases as long as $S(t) > \bar S$.\ %
Hence, for the pandemic to die out, $S_0 - \bar{S}$ many people have to get infected first.\ %
Consequently, we do not have cost controllability on~$\mathcal{A}$.\footnote{Cost controllability is satisfied at~$\bar{x}$ if there exists some $\rho > 0$ such that $V_\infty(x) \leq \rho \cdot \ell^\star(x)$ with $\ell^\star(x) : = \inf_u \ell(x,u)$ for all $x \in \mathcal{N} \cap G_\Pi$ in some neighbourhood~$\mathcal{N}$ of~$\bar{x}$, see also~\cite{esterhuizen2020recursive,coron2020model}.}\ %

\section{Conclusions}
In this paper, we considered the SEIR compartmental model with control inputs representing social distancing and quarantine measures.\ %
Based on a hard infection cap, we determined a subset in the state space, where the pandemic is contained without enforcing countermeasures.\ %
We used this set to define terminal constraints for our MPC algorithm and showed initial and, thus, recursive feasibility under mild assumptions.\ %
Our numerical simulations show that the approach is suitable to maintain the infection cap while keeping the total amount of time short where countermeasures have to be enforced.\ %
We found that in order for the pandemic to abate, sufficiently many people have to be infected first, i.e., herd immunity needs to be established.\ %
Our approach exploits that~$S$ is strictly decreasing.\ % 
Taking re-infectious or births into account requires further investigation.\ % 

\bibliographystyle{plain}
\bibliography{references}

\end{document}